\documentclass[a4paper, 10pt, twoside, notitlepage]{amsart}

\usepackage{amsmath,amscd}
\usepackage{amssymb}
\usepackage{amsthm}
\usepackage{comment}
\usepackage{graphicx}

\usepackage{mathrsfs}
\usepackage[ocgcolorlinks, linkcolor=blue]{hyperref}

\usepackage[ocgcolorlinks,linkcolor=blue]{hyperref}

\theoremstyle{plain}
\newtheorem{thm}{Theorem}[section]
\newtheorem{prop}{Proposition}[section]
\newtheorem{lem}[prop]{Lemma}
\newtheorem{cor}[prop]{Corollary}

\newtheorem{rmk}[prop]{Remark}

\numberwithin{equation}{section}
\newcommand {\R} {\mathbb{R}} 
 \newcommand {\N} {\mathbb{N}}
\newcommand {\C} {\mathbb{C}} 
\newcommand {\p} {\partial}

\newcommand{\eps}{\epsilon}

\newcommand{\ol}{\overline}

\newcommand{\re}{\mathrm{Re}}
\newcommand{\im}{\mathrm{Im}}

\newcommand{\abs}[1]{\lvert #1 \rvert}          
\newcommand{\norm}[1]{\lVert #1 \rVert}         

\title[]{Partial data inverse problems and simultaneous recovery of boundary and coefficients for semilinear elliptic equations}

\author[Lassas]{Matti Lassas}
\address{Department of Mathematics and Statistics, University of Helsinki}
\curraddr{}
\email{matti.lassas@helsinki.fi}

\author[Liimatainen]{Tony Liimatainen}
\address{Department of Mathematics and Statistics, University of Jyv\"askyl\"a}
\curraddr{}
\email{tony.liimatainen@helsinki.fi}

\author[Lin]{Yi-Hsuan Lin}
\address{Department of Mathematics and Statistics, University of Jyv\"askyl\"a}
\curraddr{}
\email{yihsuanlin3@gmail.com}

\author[Salo]{Mikko Salo}
\address{Department of Mathematics and Statistics, University of Jyv\"askyl\"a}
\curraddr{}
\email{mikko.j.salo@jyu.fi}

\begin{document}
	
	\maketitle
	
	\begin{abstract}
	We study various partial data inverse boundary value problems for the semilinear elliptic equation $\Delta u+ a(x,u)=0$ in a domain in $\R^n$ by using the higher order linearization technique introduced in~\cite{lassas2019nonlinear, FO2019semilinear}. We show that the Dirichlet-to-Neumann map of the above equation determines the Taylor series of $a(x,z)$ at $z=0$ under general assumptions on $a(x,z)$. The determination of the Taylor series can be done in parallel with the detection of an unknown cavity inside the domain or an unknown part of the boundary of the domain. The method relies on the solution of the linearized partial data Calder\'on problem \cite{ferreira2009linearized}, and implies the solution of partial data problems for certain semilinear equations $\Delta u+ a(x,u) = 0$ also proved in \cite{KU2019partial}.

		The results that we prove are in contrast to the analogous inverse problems for the linear Schr\"odinger equation. There recovering an unknown cavity (or part of the boundary) and the potential simultaneously are long-standing open problems, and the solution to the Calder\'on problem with partial data is known only in special cases when $n \geq 3$. 

		\medskip
		
		
		
	\end{abstract}

	\tableofcontents

	\section{Introduction}
	
In this paper, we extend the recent studies \cite{lassas2019nonlinear, FO2019semilinear} to various partial data inverse problems for the semilinear elliptic equation 
\[
\Delta u +a(x,u)=0 \text{ in }\Omega\subset \R^n,
\] 
for $n\geq 2$.
The proofs rely on \emph{higher order linearization}. This method reduces inverse problems for  semilinear elliptic equations to related problems for the Laplace equation, with artificial source terms produced by the nonlinear interaction, and then employs the exponential solutions introduced in \cite{calderon2006inverse} to solve these problems. Hence, one can regard the nonlinearity as a tool to solve inverse problems for elliptic equations with certain nonlinearities.

As a matter of fact, many researchers have studied inverse problems for nonlinear elliptic equations. A classical method, introduced in \cite{isakov1993uniqueness} in the parabolic case, is to show that the first linearization of the nonlinear DN map is actually the DN map of the corresponding linearized equation, and then to adapt the theory of inverse problems for linear equations. For the semilinear equation $\Delta u +a(x,u)=0$, the problem of recovering the potential $a(x,u)$ was studied in \cite{isakov1994global, VictorN, sun2010inverse, imanuvilovyamamoto_semilinear}. Further results are available for inverse problems for quasilinear elliptic equations \cite{sun1996, sun1997inverse, kang2002identification, liwang_navierstokes, munozuhlmann}, for the degenerate elliptic $p$-Laplace equation \cite{salo2012inverse, branderetal_monotonicity_plaplace}, and for the fractional semilinear Schr\"odinger equation \cite{lai2019global}. Certain inverse problems for quasilinear elliptic equations on Riemannian manifolds were considered in~\cite{lassas2018poisson}. We refer to the surveys \cite{sun2005, uhlmann2009electrical} for more details on inverse problems for nonlinear elliptic equations. 

Inverse problems for hyperbolic equations with various nonlinearities have also been studied.  Many of the results mentioned above rely on a solution to a related inverse problem for a linear equation, which is in contrast to the study of inverse problems for nonlinear hyperbolic equations. In fact, it has been realized that the nonlinearity can be beneficial in solving inverse problems for nonlinear hyperbolic equations. 

By regarding the nonlinearity as a tool, some unsolved inverse problems for hyperbolic linear equations have been solved for their nonlinear analogues. Kurylev-Lassas-Uhlmann \cite{kurylev2018inverse} studied the scalar wave equation with a quadratic nonlinearity.
In \cite{lassas2018inverse}, the authors studied inverse problems for general semilinear wave equations on Lorentzian manifolds, and in \cite{lassas2017determination} they studied similar problems for the Einstein-Maxwell equations. We also refer readers to \cite{chen2019detection,de2018nonlinear,kurylev2014einstein,wang2016quadartic} and references therein for further results on inverse problems of nonlinear hyperbolic equations.

In this work we employ the method introduced independently in \cite{lassas2019nonlinear} and \cite{FO2019semilinear} which uses nonlinearity as a tool that helps in solving inverse problems for certain nonlinear elliptic equations. The method is based on \emph{higher order linearizations} of the DN map, and essentially amounts to using sources with several parameters and obtaining new linearized equations after differentiating with respect to these parameters. The works \cite{lassas2019nonlinear, FO2019semilinear} considered inverse problems with boundary measurements on the whole boundary, also on manifolds of certain type. In this article we will consider similar problems in Euclidean domains when the data is given only on part of the boundary, or when the domain includes an unknown cavity or an unknown part of the boundary. Moreover, just before this article was submitted to arXiv, the preprint \cite{KU2019partial} of Krupchyk and Uhlmann appeared on arXiv. The work \cite{KU2019partial} considers the partial data Calder\'on problem for certain semilinear equations and proves Corollary \ref{corollary_partial_data} below.



Let us describe more precisely the semilinear equations studied in this article. Let $\Omega \subset \R^n$ be a bounded domain with $C^\infty$ boundary $\p \Omega$, where $n\geq 2$.
Consider the following second order semilinear elliptic equation
\begin{align}\label{main equation_general}
	\begin{cases}
	\Delta u + a(x,u)=0 & \text{ in }\Omega, \\
	u=f & \text{ on }\p \Omega.
	\end{cases}
\end{align}
We will assume that the boundary data satisfies $f \in C^s(\p \Omega)$ and $\norm{f}_{C^s(\p \Omega)}\leq \delta$, where $s > 1$ is not an integer and $\delta>0$ is a sufficiently small number. The space $C^s$ is the classical H\"older space. For the function $a=a(x,z)$, we assume that $a$ is $C^{\infty}$ in $\ol{\Omega} \times \R$ and satisfies one of the following conditions:

Either $a=a(x,z)$ satisfies 
\begin{align}\label{condition of a 1}
	a(x,0)=0, \text{ and } 0 \text{ is not a Dirichlet eigenvalue of }\Delta + \p_za(x,0)\text{ in }\Omega,
\end{align}
or $a=a(x,z)$ satisfies 
\begin{align}\label{condition of a}
	a(x,0)=\p_z a(x,0)=0.
\end{align}
Note that the condition \eqref{condition of a} is stronger than \eqref{condition of a 1}. Nonlinearities satisfying \eqref{condition of a} together with the condition $\p_z^k a(x,0)\neq 0$ for some $k\geq 2$ are called \emph{genuinely nonlinear} in \cite{lassas2018inverse} in the context of inverse problems of nonlinear hyperbolic equations. The benefit of assuming \eqref{condition of a} is that the linearized equation will be just the Laplace equation.

For nonlinearities satisfying \eqref{condition of a 1}, it follows from \cite[Proposition 2.1]{lassas2019nonlinear} that the boundary value problem \eqref{main equation_general} is well-posed for small boundary data $f\in C^s(\p \Omega)$. Hence, we can find a unique small solution $u$ of \eqref{main equation_general} and directly define the corresponding \emph{Dirichlet-to-Neumann} map (DN map) $\Lambda_a$ such that 
\[
\Lambda_a:C^s (\p \Omega)\to C^{s-1}(\p \Omega), \quad \Lambda_a: f\mapsto \p _\nu u|_{\p \Omega},
\]
where $\p _\nu$ is the normal derivative on the boundary $\p \Omega$.

To set the stage, we first state the full data uniqueness result that follows from the method of \cite{lassas2019nonlinear, FO2019semilinear} for Euclidean domains. This is not covered by earlier results on inverse problems for semilinear equations \cite{isakov1994global, VictorN, sun2010inverse, imanuvilovyamamoto_semilinear}, which often assume a sign condition such as $\partial_u a(x,u)\le 0$. 

\begin{thm}[Global uniqueness]\label{thm uniqueness of a}
	Let $\Omega\subset \R^n$ be a bounded domain with $C^\infty$ boundary $\p \Omega$, where $n\geq 2$. Let $a_j(x,z)$ be $C^\infty$ functions in $x,z$ satisfying \eqref{condition of a 1} for $j=1,2$. Let $\Lambda_{a_j}$ be the DN maps of 
	\begin{align*}
		\Delta u +a_j (x,u)=0 \text{ in }\Omega,
	\end{align*}
	for $j=1,2$, and assume that 
	\[
	\Lambda_{a_1}(f)=\Lambda_{a_2}(f)
	\]
	for any $f\in C^s(\p \Omega)$ with $\norm{f}_{C^s(\p \Omega)}<\delta$, where $\delta>0$ is a sufficiently small number. Then we have
	\begin{align}\label{taylor_series_agree}
		\p^k_z a_1(x,0) = \p^k_z a_2 (x,0) \text{ in }\Omega ,\text{ for }k\geq 1.
	\end{align}
\end{thm}

Theorem \ref{thm uniqueness of a} is contained in \cite{FO2019semilinear} also in the case where $a_j$ are H\"older continuous in the $x$ variable, and the case where $a(x,z)=q(x)z^m$ with $q\in C^\infty(\overline{\Omega})$, $m\in \N$ and $m\geq2$, is contained in \cite[Theorem 1.2]{lassas2019nonlinear}. To prepare for the partial data results, we will give a proof of Theorem \ref{thm uniqueness of a} as well as a reconstruction algorithm to recover the coefficients $\p_z^ka(x,0)$ for all $k \geq 2$ in Section \ref{Section 2}.

Next, we introduce an \emph{inverse obstacle problem} for semilinear elliptic equations. Let $\Omega$ and $D$ be a bounded open sets with  $C^\infty$ boundaries $\p \Omega$ and $\p D$ such that $D\subset\subset \Omega$. Assume that $\p \Omega$ and $\Omega \setminus \overline{D}$ are connected. Let $a(x,z)\in C^\infty((\overline\Omega \setminus D) \times \R)$ be a function satisfying \eqref{condition of a} for $x\in \Omega \setminus \overline{D}$. Consider the following semilinear elliptic equation 
\begin{align}\label{main equation_cavity}
	\begin{cases}
	\Delta u + a(x,u) =0 & \text{ in }\Omega \setminus \overline{D},\\
	u =0 &  \text{ on }\p D, \\
	u =f & \text{ on }\p \Omega.
	\end{cases}
\end{align}
For $s>1$ and $s \notin \N$, let $f \in C^s(\p \Omega)$ with $\norm{f}_{C^s(\p \Omega)}< \delta$, where $\delta>0$ is a sufficiently small number. The condition \eqref{condition of a} yields the well-posedness of \eqref{main equation_cavity} for small solutions by \cite[Proposition 2.1]{lassas2019nonlinear}, and one can define the corresponding DN map $\Lambda_{a}^D$, with Neumann values measured only on $\p \Omega$, by 
\[
 \Lambda_{a}^D : C^{s}(\p \Omega) \to C^{s-1}(\p \Omega),  \quad \Lambda_{a}^D: f \mapsto \p_\nu u |_{\p \Omega}.
\] 
The inverse obstacle problem is to determine the unknown cavity $D$ and the coefficient $a$ from the DN map $\Lambda_{a}^D$. Our second main result is as follows.


\begin{thm}[Simultaneous recovery: Unknown cavity and coefficients]\label{thm: Nonlinear Schiffer's problem}
	Assume that $\Omega \subset \R^n$, $n\geq 2$, is a bounded domain with connected $C^\infty$ boundary $\p \Omega$. Let $D_1, D_2\subset \subset \Omega$ be nonempty open subsets with $C^\infty$ boundaries such that $\Omega \setminus \overline{D_j}$ are connected. For $j=1,2$, let 
	\[
	a_j(x,z)\in C^\infty((\Omega\setminus \overline{D_j})\times \R)
	\]
	satisfy \eqref{condition of a} and denote by $\Lambda_{a_j}^{D_j}$ the DN maps of the following Dirichlet problems 
	\begin{align*}
	\begin{cases}
		\Delta u_j +a_j (x,u_j)=0 & \text{ in }\Omega \setminus \overline{D_j}, \\
		u_j =0 & \text{ on }\p D_j,\\
		u_j =f & \text{ on }\p \Omega
	\end{cases}
	\end{align*}
	defined with respect to the unique small solution for sufficiently small $f\in C^s(\p \Omega)$ (see \cite[Section 2]{lassas2019nonlinear} for detailed discussion).
	Assume that 
	\[
	\Lambda_{a_1}^{D_1}(f)= \Lambda_{a_2}^{D_2}(f) \text{ on }\p \Omega \text{ whenever $\norm{f}_{C^s(\p \Omega)}$ is sufficiently small.}
	\] 
	Then 
	\[
	D:=D_1 = D_2 \quad \text{ and } \quad \p_z^ka_1(x,0)=\p_z^ka_2(x,0) \text{ in }\Omega\setminus \overline{D} \text{ for }k\geq 2.
	\]
\end{thm}

The proof is based on higher order linearizations, and relies on the solution of the linearized Calder\'on problem with partial data given in \cite{ferreira2009linearized}. 

We remark that the analogous problem for the case $a(x,u)=q(x)u$ becomes an inverse problem for the linear Schr\"odinger equation. The inverse problem of determining $D$ from $\Lambda_{D,a}$ is usually regarded as  the \emph{obstacle problem}. The obstacle problem with a single measurement, i.e., determining the obstacle $D$ by a single Cauchy data $\left\{u|_{\p \Omega} ,\p_\nu u|_{\p \Omega} \right\}$ is a long-standing problem in inverse scattering theory. This type problem is also known as \emph{Schiffer's problem}, and the problem has been widely studied when the surrounding coefficients are known a priori. We refer the readers to  \cite{colton2012inverse,isakov2006inverse,liu2008uniqueness} for introduction and discussion.

Many researchers have made significant progress in recent years on Schiffer's problem for the case with general polyhedral obstacles. For the uniqueness and stability results, see \cite{AR2005determining,CY2003uniqueness,LZ2006uniqueness,LZ2007unique,rondi2003unique,rondi2008stable}. Under the assumption that $\partial D$ is nowhere analytic, Schiffer's problem was solved in \cite{HNS2013analytic}. However, Schiffer's problem still remains open for the case with general obstacles. Furthermore, a nonlocal type Schiffer's problem was solved by \cite{cao2017simultaneously}.
We also want to point out that the simultaneous recovery of an obstacle and an unknown surrounding potential is also a long-standing problem in the literature. This problem is closely related to the \emph{partial data} Calder\'on problem \cite{kenig2007calderon,imanuvilov2010calderon}. Unique recovery results in the literature are based on knowing the embedded obstacle to recover the unknown potential \cite{imanuvilov2010calderon}, knowing the surrounding potential to recover the unknown obstacle \cite{KL2013direct,KP1998recovering,LZZ2015determining,LZ2010direct,OD2006inverse}, or using multiple spectral data to recover both the obstacle and potential \cite{LL2017recovery}.

Based on the connection of simultaneous recovery problems and the partial data Calder\'on problem, we will next study a partial data problem for semilinear elliptic equations. In fact, we will consider the case where both the coefficients of the equation and a part of the boundary are unknown. In the study of partial data inverse problems for (linear) elliptic equations one usually assumes that the non-accessible part of the boundary is a priori known. This is not always a reasonable assumption in practical situations. For example, in medical imaging the body shape outside of the attached measurement device may not be precisely known. 

Let $\Omega \subset \R^n$ be a bounded connected domain with $C^\infty$ boundary $\p \Omega$. Let $\Gamma\subset \p \Omega$
be nonempty open set (the known part of the boundary), and assume that we do not know $\p \Omega \setminus \Gamma$ a priori. We consider the following semilinear elliptic equation 
\begin{align}\label{main equation2}
\begin{cases}
\Delta u + a(x,u) =0 & \text{ in }\Omega ,\\
u =0 &  \text{ on }\p \Omega\setminus \Gamma, \\
u =f & \text{ on }\Gamma,
\end{cases}
\end{align}
where $a(x,u)$ is a smooth function fulfilling \eqref{condition of a}. 
For $s>1$ and $s\notin \N$, let $f \in C^s_c(\Gamma)$ with $\norm{f}_{C^s(\Gamma)}< \delta$, where $\delta >0$ is any sufficiently small number. Then by the well-posedness of \eqref{main equation2} (see \cite[Proposition 2.1]{lassas2019nonlinear} again), one can define the corresponding DN map $\Lambda_{a}^{\Omega,\Gamma}$ with
\[
\Lambda_{a}^{\Omega,\Gamma} : C^s_c(\Gamma) \to C^{s-1}(\Gamma), \ \  f \mapsto \p_\nu u |_{\Gamma}.
\] 
The inverse problem is to determine unknown part of the boundary $\p \Omega\setminus \Gamma$ and the coefficient $a$ from the DN map $\Lambda_{a}^{\Omega,\Gamma}$. 

\begin{thm}[Simultaneous recovery: Unknown boundary and coefficients]\label{thm: partial data}
	Let $\Omega_j  \subset \R^n$, $n \geq 2$, be a bounded domain with $C^\infty$ boundary $\p \Omega_j$ for $j=1,2$, and let $\Gamma$ be a nonempty open subset of both $\p \Omega_1$ and $\p \Omega_2$. Let $a_j(x,z)$ be smooth functions satisfying \eqref{condition of a}. Let $\Lambda_{a_j}^{\Omega_j, \Gamma}$ be the DN maps of the following problems 
	\begin{align*}
	\begin{cases}
	\Delta u +a_j(x,u)=0 & \text{ in }\Omega_j, \\
	u_j =0 & \text{ on }\p \Omega_j \setminus \Gamma,\\
	u_j =f & \text{ on }\Gamma,
	\end{cases}
	\end{align*}
	for $j=1,2$. Assume that 
	$$
	\Lambda_{a_1}^{\Omega_1,\Gamma}(f)= \Lambda_{a_2}^{\Omega_2, \Gamma}(f) \text{ on }\Gamma
	$$ 
	for any $f\in C^s_c(\Gamma)$ with $\norm{f}_{C^s(\Gamma)}<\delta$, for a sufficiently small number $\delta>0$. Then we have 
	\[
	\Omega_1 = \Omega_2:=\Omega \quad \text{ and }\quad \p_z^k a_1(x,0)=\p_z^ka_2(x,0) \text{ in }\Omega \text{ for } k\geq 2.
	\]
\end{thm}

The proof again relies on higher order linearizations and on the solution of the linearized Calder\'on problem with partial data \cite{ferreira2009linearized}. By using Theorem \ref{thm: partial data}, we immediately have the following result, which was first proved in the preprint \cite{KU2019partial} that appeared on arXiv just before this preprint was submitted.

\begin{cor}[Partial data] \label{corollary_partial_data}
	Let $\Omega \subset \R^n$, $n \geq 2$, be a bounded domain with $C^\infty$ boundary $\p \Omega$, and let $\Gamma \subset \Omega$ be a nonempty open subset. Let $a_j(x,z)$ be smooth functions satisfying \eqref{condition of a} and let $\Lambda_{a_j}^{\Omega,\Gamma}$ be  the partial data DN map for the Dirichlet problem 
	\begin{align*}
	\begin{cases}
	\Delta u+ a_j(x,u)=0 & \text{ in }\Omega, \\
	u=0 & \text{ on }\p \Omega \setminus \Gamma, \\
	u=f  & \text{ on }\Gamma,
	\end{cases}
	\end{align*}
	for $j=1,2$. Assume that 
	\[
	\Lambda_{a_1}^{\Omega,\Gamma}(f)=\Lambda_{a_2}^{\Omega,\Gamma}(f) \text{ on }\Gamma,
	\]
	for any sufficiently small $f\in C^s_c(\Gamma)$. Then
	\[
	\p_z^k a_1(x,0)=\p_z^k a_2(x,0) \text{ in }\Omega \text{ for }k\geq 2. 
	\]
\end{cor}

For the corresponding linear equation, i.e., $a(x,u)=q(x)u$, the partial data problem of determining $q$ from the DN map $\Lambda_q^{\Omega,\Gamma}(f)|_{\Gamma}$ for any $f$ supported in $\Gamma$, where $\Gamma$ is an arbitrary nonempty open subset of $\p \Omega$, was solved in \cite{imanuvilov2010calderon} for $n=2$ and $q_j\in C^{2,\alpha}$. For $n\geq 3$, the partial data problem stays open, but there are partial results  \cite{bukhgeim2002partial, kenig2007calderon,isakov2007uniqueness,kenig2014calderon} when $\partial \Omega$ is assumed to be known. We refer to the surveys \cite{imanuvilov2013uniqueness,kenig2014recent} for further references.

\begin{rmk}
	In this work, we do not pursue optimal regularity assumptions for our inverse problems. Instead, we want to demonstrate how the nonlinearity helps us in understanding related inverse problems. In addition, if we assume that $a_j(x,z)$ are real analytic in $z$ for $j=1,2$, then one can completely recover the nonlinearity and show that $a_1(x,z)=a_2(x,z)$ in Theorems \ref{thm uniqueness of a}, \ref{thm: Nonlinear Schiffer's problem} and \ref{thm: partial data}. In particular this applies to equations of the type $\Delta u + q(x) u^m = 0$, where $m \geq 2$ is an integer.
\end{rmk}

The paper is structured as follows. In Section \ref{Section 2} we prove Theorem \ref{thm uniqueness of a}. We also provide reconstruction algorithms for $\p_z^ka (x,z)|_{z=0}$ for all $k\geq 2$. Theorem \ref{thm: Nonlinear Schiffer's problem} and Theorem \ref{thm: partial data} will be proved in Section \ref{Section 3} and Section \ref{Section 4}, respectively. Appendix \ref{sec_appendix} contains the proof of a topological lemma required in the arguments.

\vspace{10pt}

\noindent {\bf Acknowledgements.} 
All authors were supported by the Finnish Centre of Excellence in Inverse Modelling and Imaging (Academy of Finland grant 284715). M.S.\ was also supported by the Academy of Finland (grant 309963) and by the European Research Council under Horizon 2020 (ERC CoG 770924).

\section{Proof of Theorem \ref{thm uniqueness of a}}\label{Section 2}
We use higher order linearizations to prove Theorem \ref{thm uniqueness of a}. Before the proof we recall  Calder\'on's exponential solutions (\cite{calderon2006inverse}) to the equation $\Delta v =0$ in $\R^n$, and the complex geometrical optics solutions (CGOs) that solve $\Delta v + qv=0$ on a domain $\Omega$ in $\R^n$. These solutions will be used in the proof of Theorem \ref{thm uniqueness of a}. The exponential solutions of Calder\'on are of the form
\begin{align}\label{calderon_exponential}
v_1(x):=\exp((\eta +i\xi)\cdot x), \quad v_2(x):=\exp((-\eta+i\xi)\cdot x),
\end{align}
where $\eta$ and $\xi$ are any vectors in $\R^n$ that satisfy $\eta\perp \xi$ and $\abs{\eta}=\abs{\xi}$.
The functions $v_1$ and $v_2$ solve the Laplace equation
\[
\Delta v_1 =\Delta v_2 =0 \text{ in }\R^n.
\]
The linear span of the products $v_1 v_2 =\exp(2i\xi \cdot x)$, $\xi\in \R^n$, of Calder\'on's exponential solutions forms a dense set in $L^1(\Omega)$. In particular, if
\[
\int_\Omega f v_1 v_2\, dx=0
\]
holds for all Calder\'on's exponential solutions $v_1$ and $v_2$, then $f=0$.

The complex geometrical optics solutions (CGOs) generalize Calder\'on's exponential solutions. For $n\geq 3$, they are of the form (see e.g.~\cite{sylvester1987global}) 
\begin{equation}\label{cgos}
 v_1(x)=e^{\rho_1 \cdot x} (1+r_1), \quad v_2(x)=e^{\rho _2\cdot x}(1+r_2),
\end{equation}
where $\rho_1=\eta+i\left(\xi + \zeta\right) \in \C^n$ and $\rho_2=-\eta+i\left(\xi - \zeta \right) \in \C^n$. Here $\eta, \xi, \zeta \in \R^n$ satisfy 
\[
\eta \cdot \xi = \xi \cdot \zeta = \zeta \cdot \eta =0, \text{ and }|\eta|^2=|\xi |^2 +|\zeta|^2.
\]
The idea is that $\xi$ is fixed but $\abs{\eta}, \abs{\zeta} \to \infty$. If $q\in L^\infty$, the CGO solutions $v_1$ and $v_2$ satisfy
\[
 (\Delta +q)v_1=(\Delta +q)v_2=0 \text{ in }\Omega
\]
and $\norm{r_j}_{L^2(\Omega)}\leq \frac{C}{|\rho_j|}$ for some constant $C>0$ depending on $q_j$, for $j=1,2$. Thus the product $v_1 v_2$ converges to $e^{2ix \cdot \xi}$ as $\abs{\eta}, \abs{\zeta} \to \infty$. For $n=2$ one needs to use CGOs with quadratic phase functions instead, see~\cite{bukhgeim2008recovering} (for the conductivity equation CGOs with linear phase functions are still useful \cite{nachman1996global, astala2006calderon}).

The products of pairs of CGOs form a complete set in $L^1(\Omega)$ by \cite{sylvester1987global} for $n \geq 3$ and in $L^2(\Omega)$ by \cite{bukhgeim2008recovering, blasten2017singular} for $n=2$. In particular, if $f \in L^{\infty}(\Omega)$ and 
\[
 \int_\Omega f v_1 v_2\, dx=0
\]
holds for all CGOs $v_1$ and $v_2$, then $f=0$. 
We refer to the survey~\cite{uhlmann2009electrical} for more details and references on CGOs.

Before the proof, we need to discuss a minor issue: the equation $\Delta u + a(x,u) = 0$ involves real valued solutions ($a$ is defined on $\ol{\Omega} \times \R$), whereas exponential solutions and CGOs are complex valued. However, in the proof we can just use the real and imaginary parts of these solutions (which are solutions themselves, since the coefficients are real valued) by virtue of the following simple lemma.

\begin{lem} \label{lemma_complex_real_solutions}
Let $f \in L^{\infty}(\Omega)$, $v_1, v_2 \in L^2(\Omega)$, and $v_3, \ldots, v_m \in L^{\infty}(\Omega)$ be complex valued functions where $m \geq 2$. Then 
\[
\int_{\Omega} f v_1 \cdots v_m \,dx = \sum_{j=1}^{2^m} \int_{\Omega} c_j f w_1^{(j)} \cdots w_m^{(j)} \,dx
\]
where $c_j\in \{\pm 1, \pm i \} $ and $w_1^{(j)}\in \{\mathrm{Re}(v_1), \mathrm{Im}(v_1) \}, \cdots, w_m^{(j)}\in \{\mathrm{Re}(v_m), \mathrm{Im}(v_m) \}$ for $1\leq j \leq 2^m$.
\end{lem}
\begin{proof}
The result follows by writing 
\[
\int_{\Omega} f v_1 \cdots v_m \,dx = \int_{\Omega} f (\re(v_1) + i \im(v_1)) \cdots (\re(v_m) + i \im(v_m)) \,dx 
\]
and by multiplying out the right hand side.
\end{proof}


\begin{proof}[Proof of  Theorem \ref{thm uniqueness of a}] We split the proof into two parts, where in the first part we assume that the linear terms of the operators $\Delta +a_j(x,z)$ vanish: $\p_z a_j(x,0)\equiv 0$, $j=1,2$.
The proof in this case is based on Calder\'on's exponential solutions. In the second part we consider the case $\p_z a_j(x,0)\neq 0$ and use CGOs instead of Calder\'on's exponential solutions.

	\vspace{10pt}

{\it Case 1. $\p_z a_j(x,0)\equiv 0$.}\\

\noindent The proof is by induction on the order of the order of differentiation $k\in \N$. By assumption, we have that
\[
 \p_za_1(x,0)=0=\p_za_2(x,0).
\]
Let then $N\in \N$ and assume that 
\begin{equation}\label{inducition_assumption}
 \p_z^ka_1(x,0)=\p_z^ka_2(x,0) \text{ for all }k=1,2, \cdots, N.
\end{equation}
The induction step is to show that~\eqref{inducition_assumption} holds for $k=N+1$. 

For $\ell=1,\ldots,N+1$, let $\eps_\ell$ be small positive real numbers, and let $f_\ell \in C^s(\p \Omega)$ be functions on the boundary. Let us denote $\epsilon=(\eps_1,\eps_2,\ldots, \eps_{N+1})$ and let
%
%
the function 
\[
u_j:=u_j(x;\epsilon), \quad j=1,2,
\]
be the unique small solution of the Dirichlet problem
	\begin{align}\label{equ1a1 in 1st example}
	\begin{cases}
	\Delta u_j +a_j(x,u_j) =0 &\text{ in }\Omega,\\
	u_j =\sum _{\ell=1}^{N+1}\eps_\ell f_\ell  &\text{ on }\p \Omega.
	\end{cases}
	\end{align}
	The existence of the unique small solution is guaranteed by~\cite[Proposition 2.1]{lassas2019nonlinear} (by redefining $\eps_\ell$ to be smaller if necessary). To prove the induction step, we will differentiate the equation~\eqref{equ1a1 in 1st example} with respect to the $\eps_\ell$ parameters several times. The differentiation is justified by~\cite[Proposition 2.1]{lassas2019nonlinear}.
	
	We begin with the first order linearization as follows. 
	Let us differentiate \eqref{equ1a1 in 1st example} with respect to $\epsilon_{\ell}$, so that 
	\begin{align}\label{equ2a1 in 1st example}
	\begin{cases}
	\Delta \left(\frac{\p}{\p \eps_\ell}u_j\right) + \p_{z} a_j (x,u_j)\left(\frac{\p}{\p \eps_\ell}u_j\right)=0 & \text{ in }\Omega, \\
	\frac{\p}{\p \eps_\ell}u_j = f_\ell &\text{ on }\p \Omega.
	\end{cases}
	\end{align}
	Evaluating \eqref{equ2a1 in 1st example} at $\epsilon = 0$ shows that 
	\begin{align*}
	\Delta v_j^{(\ell)}=0 \text{ in }\Omega \text{ with }v_j^{(\ell)}=f_\ell \text{ on }\p \Omega,
	\end{align*}
	where 
	\[
	v_j^{(\ell)}(x) =\left.\frac{\p}{\p \eps_\ell}\Big|_{\epsilon=0}u_j(x;\epsilon)\right..
	\]
	Here we have used $u_j(x;\epsilon)|_{\epsilon =0 }\equiv 0$ so that $\p_{z }a_j(x,u_j)|_{\eps=0}\equiv0$ in $\Omega$.  
	The functions $v_j^\ell$ are harmonic functions defined in $\Omega$ with boundary data $f_\ell|_{\p \Omega}$.
	By uniqueness of the Dirichlet problem for the Laplace operator we have that 
	\begin{align}\label{v_1 =v_2a1 Rn}
	v^{(\ell)}:=v_1^{(\ell)}=v_2^{(\ell)} \text{ in } \Omega \quad  \text{ for } \ell =1,2,\cdots, N+1.
	\end{align}
	
	For illustrative purposes we show next how to prove that $\p_z^2a_1(x,0)=\p_z^2a_2(x,0)$, which corresponds to the special case $N=1$. The second order linearization is given by differentiating \eqref{equ2a1 in 1st example} with respect to $\epsilon_k$ for arbitrary $k\neq \ell$ where $k,\ell  \in \{ 1,2,\cdots,N+1 \}$. Doing so yields
	\begin{align}\label{equ 3a1 in 1st example}
	\begin{cases}
	\Delta \left(\frac{\p^2}{\p \eps_k \p \eps_\ell}u_j \right) + \p_{z}a_j (x,u_j) \left(\frac{\p^2}{\p \eps_k \p \eps_\ell}u_j \right) + \p ^2_{z}a(x,u) \left(\frac{\p u_j}{\p \eps_k}\right) \left(\frac{\p u_j }{\p \eps_\ell } \right) =0 & \text{ in }\Omega, \\
	\frac{\p^2}{\p \eps_k \p \eps_\ell}u_j =0 & \text{ on }\p \Omega.
	\end{cases}
	\end{align}
	By evaluating~\eqref{equ 3a1 in 1st example} at $\epsilon=0$ we have that 
	\begin{align}\label{equ 4a1 in 1st example}
	\begin{cases}
	\Delta  w_j^{(k\ell )}+ \p ^2_{z}a_j(x,0) v^{(k)}v^{(\ell)}  =0 & \text{ in }\Omega, \\
	w_j^{(k\ell)}=0 & \text{ on }\p \Omega,
	\end{cases}
	\end{align}
	where we have denoted $w_j^{(k\ell )}(x)=\frac{\p^2}{\p \eps_k \p \eps_\ell}u_j(x;\epsilon)\big|_{\epsilon =0 }$ and used $u_j(x;\epsilon)|_{\epsilon=0}\equiv 0$ in $\Omega$ for $j=1,2$. 
	By using the fact that $\Lambda_{a_1}\left(\sum_{\ell=1}^{N+1} \eps_\ell f_\ell\right)=\Lambda_{a_2}\left(\sum_{\ell=1}^{N+1} \eps_\ell f_\ell\right)$,
	we have that
	\begin{equation}\label{normal_derivatives}
	\p_{\nu} u_1|_{\p \Omega} = \p_{\nu} u_2|_{\p \Omega}.
	\end{equation}
	By applying $\p_{\eps_k} \p_{\eps_\ell}|_{\epsilon=0}$ to the equation~\eqref{normal_derivatives} above shows that 
	\[
	\left.\p_\nu w_1^{(k\ell)}\right|_{\p \Omega}=\left.\p_\nu w_2^{(k\ell)}\right|_{\p \Omega}, \text{ for } k, \ell=1,\ldots,N+1.
	\]
	(We remind that this formal looking calculation is justified by~\cite[Proposition 2.1]{lassas2019nonlinear}.)
	Hence, by integrating the equation~\eqref{equ 4a1 in 1st example} over $\Omega$ 
	and by using integration by parts we obtain the equation
	\begin{align}\label{2nd ordera1 integral id in Rn}
	0=&\notag \int_{\p \Omega} \left(\p_\nu w_1^{(k\ell)}-\p_\nu w_2^{(k\ell)}\right) \,dS =\int_{\Omega} \Delta\left(w_1^{(k\ell)} -w_2^{(k\ell)}\right) \,dx\\
	=& \int_\Omega \left( \p ^2_za_2(x,0)-\p^2_z a_1(x,0)\right)v^{(k)}v^{(\ell)} \,dx
	\end{align}
	where $v^{(k)}$ and $v^{(\ell)}$ are defined in \eqref{v_1 =v_2a1 Rn}. (More generally, as in \cite{lassas2019nonlinear} we could as well have integrated against a third harmonic function $v^{(m)}$.) Therefore, by choosing $f_k$ and $f_{\ell}$ as the boundary values of the real or imaginary parts of Calder\'on's exponential solutions $v_1$ and $v_2$ in~\eqref{calderon_exponential} (note that the real and imaginary parts of $v_1$ and $v_2$ are also harmonic), and by using Lemma \ref{lemma_complex_real_solutions}, we obtain that 
	\[
	\int_\Omega \left( \p ^2_za_2(x,0)-\p^2_z a_1(x,0)\right) v_1 v_2 \,dx = 0.
	\]
	It follows that the Fourier transform of the difference $\p^2_z a_1(x,0)-\p^2_z a_2 (x,0)$ is zero. Thus $\p^2_z a_1(x,0)=\p^2_z a_2 (x,0)$. 
	We define
	\begin{align}\label{2nd ordera1 recovery}
	\p_z^2 a(x,0):= \p^2_z a_1(x,0)=\p^2_z a_2 (x,0).
	\end{align} 
	
	We also note that by using~\eqref{2nd ordera1 recovery}, the equation \eqref{equ 4a1 in 1st example} shows that the function $w_1^{(k\ell)}-w_2^{(k\ell)}$ solves 
	\[
	\Delta \left(w_1^{(k\ell)}-w_2^{(k\ell)} \right)=0, \text{ with }w_1^{(k\ell)}-w_2^{(k\ell)}=0 \text{ on }\p \Omega.
	\]
	Thus we have that
	\begin{align}\label{2nd ordera1 equal solutions}
	w^{(k\ell)}:=w_1^{(k\ell)}=w_2^{(k\ell)} \text{ in } \Omega.
	\end{align}

	We have now shown how to prove the special case $N=1$. Let us return to the general case $N\in \N$. To prove the general case, we first show by induction within induction, call it subinduction, that
	\begin{equation}\label{subinduction}
	 \frac{\p^{k} u_1(x;0)}{\p \eps_{\ell_1}\cdots \p \eps_{\ell_{k}}} =\frac{\p ^{k} u_2(x;0)}{\p \eps_{\ell_1}\cdots \p \eps_{\ell_{k}}} \text{ in }\Omega,
	\end{equation}
	for all $k=1,\ldots,N$. The claim holds for $k=1$ by~\eqref{v_1 =v_2a1 Rn}. Let us then assume that~\eqref{subinduction} holds for all $k\leq K<N$. The linearization of order $K+1$ evaluated at $\eps=0$ reads
	\begin{align}\label{equ 7a1 in 1st example}
	& \Delta \left(\frac{\p^{K+1} u_j(x,0)}{\p \eps_{\ell_1}\cdots \p \eps_{\ell_{K+1}}} \right) +  
	R_{K}(u_j,a_j,0)+ \p_z^{K+1} a_j(x,0)\left( \Pi _{k =1}^{K+1} v^{(\ell_k)} \right)=0 \text{ in }\Omega,
	\end{align}
	where 
	$R_K(u_j,a_j,0)$ is a polynomial of the functions $\p_z ^{k}a_j(x,0)$ and $ \frac{\p^{k} u_j(x;0)}{\p \eps_{\ell_1}\cdots \p \eps_{\ell_{k}}}$ for all $k\leq K$.
	By the induction assumptions~\eqref{inducition_assumption} and~\eqref{subinduction} these functions agree for $j=1,2$. Thus it follows that
    \begin{align*}
    \begin{cases}
    \Delta \left(\p^{K+1}_{\eps_{\ell_1}\cdots \eps_{\ell_{K+1}}}u_1(x,0)- \p^{K+1}_{\eps_{\ell_1}\cdots \eps_{\ell_{K+1}}}u_2(x,0) \right)=0 & \text{ in } \Omega \\
    \p^{K+1}_{\eps_{\ell_1}\cdots \eps_{\ell_{K+1}}}u_1(x,0)- \p^{K+1}_{\eps_{\ell_1}\cdots \eps_{\ell_{K+1}}}u_2(x,0)=0 &\text{ on } \p \Omega.
    \end{cases}
	\end{align*}
	(Above we have used the abbreviation $\p^{K+1}_{\eps_{\ell_1}\cdots \eps_{\ell_{K+1}}}u_j(x,0)=\frac{\p^{K+1} u_j(x,0)}{\p \eps_{\ell_1}\cdots \p \eps_{\ell_{K+1}}}$ for $j=1,2$ and for $K \in \N$, which will also be used later in the proof).
    Thus by the uniqueness of solutions to the Laplace equation we have that~\eqref{subinduction} holds for $k=1,\ldots,K+1$, which concludes the induction step of the subinduction. Thus \eqref{subinduction} holds for all $k=1,\ldots,N$. 
	

	Let us then continue with the main induction argument of the proof. The linearization of order $N+1$ at $\eps=0$ yields the equation~\eqref{equ 7a1 in 1st example} with $N$ in place of $K$. By the subinduction, we have that $R_N(u_1,a_1,0)=R_N(u_2,a_2,0)$. 
	By using this fact, it follows by subtracting the equations \eqref{equ 7a1 in 1st example} with $j=1$ and $j=2$ from each other (with $K=N$) that 
	\begin{align*}
	\int_\Omega \left(\p_z^{N+1} a_1(x,0)-\p_z^{N+1} a_2(x,0) \right) \left( \Pi _{k=1}^{N+1} v^{(\ell_k)} \right)dx=0.
	\end{align*}
	Here we used integration by parts and the assumption $\Lambda_{a_1}=\Lambda_{a_2}$. 
	We choose two of the functions $v^{(\ell_k)}$ to be the real or imaginary parts of the exponential solutions \eqref{calderon_exponential}, and the remaining $N-1$ of them to be the constant function $1$. Using Lemma \ref{lemma_complex_real_solutions} again, it follows that $\p_z^{N+1} a_1(x,0)= \p_z^{N+1} a_2 (x,0)$ in $\Omega$ as desired. This concludes the main induction step. 
	
	\vspace{10pt}
	
		\vspace{10pt}
	
	{\it Case 2. $\p_z a_j(x,0)\not\equiv 0$.}\\
	
	\noindent The proof is similar to the Case 1, and therefore we keep exposition short. As said before, the main difference is that we use CGOs~\eqref{cgos} instead of Calder\'on's exponential solutions~\eqref{calderon_exponential}. 
	We consider $\eps_\ell$ to be small numbers, $\ell =1,2,\cdots ,N+1$, and $\eps=(\eps_1,\cdots,\eps_{N+1})$ and $f_\ell \in C^s(\p \Omega)$, for all $\ell=1,2,\cdots,N+1$. Let the function $u_j:=u_j(x;\eps)$ be the unique small solution of 
\begin{align*}
\begin{cases}
\Delta u_j +a_j(x,u_j) =0 &\text{ in }\Omega,\\
u_j =\sum _{\ell=1}^{N+1} \eps_\ell f_\ell  &\text{ on }\p \Omega,
\end{cases}
\end{align*}
for $j=1,2$. We begin with the first order linearization as follows, which at $\eps=0$ yields: 
\begin{align}\label{equ3a in 1st example}
\begin{cases}
\left( \Delta +\p_z a_j(x,0)\right) v_j^{(\ell)}=0 & \text{ in }\Omega, \\
v_j^{(\ell)}=f_\ell & \text{ on }\p \Omega,
\end{cases}
\end{align}
where 
\[
v_j^{(\ell)}(x) =\left.\frac{\p}{\p \eps_\ell}\Big|_{\eps=0}u_j(x;\eps)\right..
\]
The functions $v_j^\ell$ are the solutions of the Schr\"odinger equation with potential $\p_z a_j(x,0)$ in $\Omega$ with boundary data $f_\ell|_{\p \Omega}$.

We show that $\p_z a_1(x,0)= \p_za_2(x,0)$ for $x\in \Omega$. Since the DN maps $\Lambda_{a_1}$ and $\Lambda_{a_2}$ agree, we have by \cite[Proposition 2.1]{lassas2019nonlinear} that the DN maps corresponding to the equation \eqref{equ3a in 1st example} are the same. It follows that 
\begin{equation}\label{first_ord_det}
\p_z a_1(x,0)= \p_za_2(x,0) 
\end{equation}
by the results \cite{bukhgeim2008recovering} and \cite{sylvester1987global} for $n=2$ and $n\geq 3$ respectively. Moreover, by using~\eqref{first_ord_det} and the uniqueness of solutions to the Dirichlet problem \eqref{equ3a in 1st example}, we have that 
\begin{align}\label{v_1 =v_2 Rn}
v^{(\ell)}:=v_1^{(\ell)}=v_2^{(\ell)} \text{ in } \Omega  \text{ for } \ell =1,2,\cdots, N+1,
\end{align} 
and we simply denote 
\[
q(x):=\p_za_1(x,0)=\p_za_2(x,0) \text{ for }x \in \Omega.
\]
Here we used the assumption~\eqref{condition of a 1}, which says that operators $\Delta +\p_z a_j(x,0)$ are injective on $H_0^1(\Omega)$, , $j=1,2$.

Since $\p_za_1(x,0)=\p_za_2(x,0)$, we have that the claim~\eqref{taylor_series_agree} of the theorem holds for $k=1$. We proceed by induction on $k$. To do that, we assume that~\eqref{taylor_series_agree} holds for all $k=1,\ldots,N$. Again, we do the $N=1$ case separately to explain how the induction works. The second order linearization yields the equations for $j=1,2$: 
\begin{align}\label{equ 4 in 1st example}
\begin{cases}
\Delta  w_j^{(k\ell )}+ q(x) w_j^{(k\ell)}+\p ^2_{z}a_j(x,0) v^{(k)}v^{(\ell)}  =0 & \text{ in }\Omega, \\
w_j^{(k\ell)}=0 & \text{ on }\p \Omega,
\end{cases}
\end{align}
where $w_j^{(k\ell )}(x)=\left.\frac{\p^2}{\p \eps_k \p \eps_\ell}u_j(x;\eps)\right|_{\eps =0 }$ and we used $u_j(x;\eps)|_{\eps=0}\equiv 0$ in $\Omega$. 
Since $\Lambda_{a_1}=\Lambda_{a_2}$, 
we have (as in Case 1) that
\[
\left.\p_\nu w_1^{(k\ell)}\right|_{\p \Omega}=\left.\p_\nu w_2^{(k\ell)}\right|_{\p \Omega}, \text{ for } k, \ell \in 1,\ldots, N.
\]

Fix $x_0\in \Omega$. We claim that there exists a solution $v^{(0)}\in H^s(\Omega)$, where $s$ can be chosen arbitrarily large, of the Schr\"odinger equation 
\begin{align}\label{equation of v^0}
	\Delta v^{(0)}+q(x)v^{(0)}=0 \text{ in } \Omega  
\end{align}
with
\[
 v^{(0)}(x_0) \neq 0.
\]
By the Runge approximation property (see e.g.\ \cite[Proposition A.2]{lassas2018poisson}), it is enough to construct such a solution in some small neighborhood $U$ of $x_0$. Since $q$ is smooth, by a perturbation argument it is enough to construct a nonvanishing solution of $\Delta w + q(x_0) w = 0$ near $x_0$. Writing $q(x_0) = \lambda^2$ for some complex number $\lambda$, it is enough to take $w = e^{i\lambda x_1}$. This completes the construction of $v^{(0)}$.

Now, multiplying \eqref{equ 4 in 1st example} by $v^{(0)}$ and integrating by parts yields that 
\begin{align}\label{2nd order integral id in Rn}
\begin{split}
0&=\int_{\p \Omega} v^{(0)} \p_\nu \left(w_1^{(k\ell)}-w_2^{(k\ell)}\right)\,dS\\
& =\int_\Omega v^{(0)} \Delta \left(w_1^{(k\ell)}-w_2^{(k\ell)}\right)\,dx+ \int_\Omega  \nabla v^{(0)}\cdot \nabla \left(w_1^{(k\ell)}-w_2^{(k\ell)}\right)\,dx\\
& =\int_\Omega q(x) (w_2^{(k\ell)}-w_1^{(k\ell)})v^{(0)}\,dx+\int_\Omega \left(\p^2_{z}a_2-\p^2_{z}a_1\right) v^{(k)}v^{(\ell)}v^{(0)}\,dx \\
& \qquad-\int_\Omega   \left(w_1^{(k\ell)}-w_2^{(k\ell)}\right)\Delta v^{(0)} \,dx \\
&=\int_\Omega \left( \p ^2_za_2(x,0)-\p^2_z a_1(x,0)\right)v^{(k)}v^{(\ell)} v^{(0)}\,dx.
\end{split}
\end{align}
Here $v^{(k)}$ and $v^{(\ell)}$ are solutions to  \eqref{v_1 =v_2 Rn},
which we now choose specifically to be real or imaginary parts of the CGOs (since $q$ is real valued, the real and imaginary parts of CGOs are also solutions of $\Delta v + q v = 0$). Then, by using Lemma \ref{lemma_complex_real_solutions} we can reduce to the case where $v^{(k)}$ and $v^{(\ell)}$ are the actual complex valued CGOs, and by applying the completeness of products of pairs of CGOs  \cite{bukhgeim2008recovering,sylvester1987global} we obtain that 
\[
\p ^2_za_2(x,0)v^{(0)}(x)=\p^2_z a_1(x,0)v^{(0)}(x) \text{ for }x \in \Omega.
\]
In particular, when $x=x_0$, we have $\p ^2_za_2(x_0,0)=\p^2_z a_1(x_0,0)$ since $v^{(0)}(x_0) \neq 0 $. Since $x_0 \in \Omega$ was arbitrary, we have that
\begin{align}\label{2nd order recovery}
\p_z^2 a(x,0):= \p^2_z a_1(x,0)=\p^2_z a_2 (x,0) \text{ for }x\in \Omega.
\end{align}
This concludes the induction step in the special case $N=1$. We also have from \eqref{equ 4 in 1st example} and \eqref{2nd order recovery} that $w_1^{(k\ell)}-w_2^{(k\ell)}$ solves 
\begin{align*}
	\begin{cases}
	\Delta \left(w_1^{(k\ell)}-w_2^{(k\ell)} \right)+q(x)\left(w_1^{(k\ell)}-w_2^{(k\ell)} \right)=0 & \text{ in }\Omega, \\
	w_1^{(k\ell)}-w_2^{(k\ell)}=0 & \text{ on }\p \Omega,
	\end{cases}
\end{align*}
then the uniqueness of the solution to the Schr\"odinger equation yields that 
\begin{align*}
w^{(k\ell)}:=w_1^{(k\ell)}=w_2^{(k\ell)} \text{ in } \Omega.
\end{align*}

Let us return to general case $N\in \N$. As in the Case 1, we first prove by the subinduction that
\begin{equation*}
	 \p^{k}_{\ell_1\cdots \ell_{k}}u_1(x;0) =\p^{k}_{\ell_1\cdots \ell_{k}}u_2(x;0) \text{ in }\Omega,
	\end{equation*}
	for all $k\leq N$. Then the linearization for $j=1,2$ of order $N+1$ shows that 
	\begin{align}\label{equ 7a1 in 1st exampleCase2}
	& (\Delta +q)\Big(\p^{N+1}_{\eps_{\ell_1}\cdots \eps_{\ell_{N+1}}}u_j(x,0)\Big) +  
	R_{N}(u_j,a_j,0)+ \p_z^{N+1} a_j(x,0)\left( \Pi _{k =1}^{N+1} v^{(\ell_k)} \right)=0,
	\end{align}
	for $x\in \Omega$, and where $R_N(u_j,a_j,0)$ is a polynomial of the functions $\p_z ^{k}a_j(x,0)$ and $\p^{k}_{\eps_{\ell_1}\cdots \eps_{\ell_{k}}}u_j(x;0)$ 
	for $k\leq N$. By the subinduction we have that $R_{N}(u_1,a_1,0)=R_{N}(u_2,a_2,0)$.
	
Finally, by multiplying~\eqref{equ 7a1 in 1st exampleCase2} by $v^{(0)}$ and repeating an integration by parts argument similar to that in~\eqref{2nd order integral id in Rn} 
shows that we have the following integral identity
\begin{align*}
\int_\Omega \left(\p_z^{N+1} a_1(x,0)-\p_z^{N+1} a_2(x,0) \right) \left( \Pi _{k =1}^{N+1} v^{(\ell_k)} \right)v^{(0)}dx=0.
\end{align*}
With the help of Lemma \ref{lemma_complex_real_solutions} we can choose $v^{(\ell_1)}$ and $v^{(\ell_2)}$ to be the CGOs as before, and we choose the remaining $N-1$ solutions as $v^{(\ell_3)}=\cdots =v^{(\ell_{N+1})}= v^{(0)}$, where $v^{(0)}$ is the solution in \eqref{equation of v^0}.
We conclude that $\p_z^{N+1} a_1(x_0,0)= \p_z^{N+1} a_2 (x_0,0)$. Since $x_0\in \Omega$ was arbitrary, we obtain that $\p_z^{N+1} a_1(x,0)= \p_z^{N+1} a_2 (x,0)$ in $\Omega$. This concludes the proof.
\end{proof}

\begin{rmk}
	In the proof of Theorem \ref{thm uniqueness of a}, we have used the Runge approximation property to construct solutions to the Schr\"odinger equation that are nonzero at a given point $x_0$. An alternative method is to construct a nonvanishing solution of $\Delta v + q(x)v=0$. This can be done by considering a complex geometrical optics solution 
	\[
	v(x)=e^{\rho \cdot x } (1+r) \text{ in }\Omega,
	\]
	where $\rho \in \C^n$. Then $r$ solves 
	\[
	e^{-\rho \cdot x}(\Delta + q)e^{\rho\cdot x}r=-q \text{ in }\Omega,
	\]
	with the estimate (see \cite[Theorem 1.1]{sylvester1987global}, the argument applies also in our case when $n \geq 2$) 
	\[
	\norm{r}_{H^s(\Omega)}\leq \frac{C}{|\rho|}\norm{q}_{H^s(\Omega)},
	\]
	for $s>n/2$ and $|\rho|$ large enough. Then by the Sobolev embedding we have that 
	\[
	\norm{r}_{L^\infty(\Omega)} \leq \frac{1}{2},
	\]
	for $\abs{\rho}$ large enough. This implies that $v(x)$ is nonvanishing in $\Omega$, and the solution $v^{(0)}$ in the proof of Theorem \ref{thm uniqueness of a} could be replaced by $v$ here. 
	
	In the case $q(x)\leq 0$ in $\Omega$  (with the sign convention $\Delta=\sum_{k=1}^n \p_{x_k}^2$), another alternative is to apply the maximum principle to construct a positive solution to the Schr\"odinger equation in $\Omega$. 
\end{rmk}


Furthermore, when the coefficient $a=a(x,z)$ of the operator $\Delta + a(x,\cdot)$ satisfies 
\begin{equation*}
\p_z a(x,0)\equiv 0 
\end{equation*}
one has the following reconstruction result.
\begin{thm}[Reconstruction] \label{thm of reconstruction}
	Let $n\geq 2$, and let $\Omega\subset \R^n$ be a bounded domain with $C^\infty$ boundary $\p \Omega$. Let $\Lambda_a$ be the DN map of the equation 
		\begin{equation*}
		\Delta u + a(x,u)=0 \text{ in } \Omega,
		\end{equation*}
		and assume that $a(x,z)\in C^\infty(\overline{\Omega} \times \R)$ satisfies \eqref{condition of a}.
		Then we can reconstruct $\p_z^k a(x,0)$ from the knowledge of $\Lambda_a$, for all $k\geq 2$.
\end{thm}
\begin{proof}
	
  For $k=2$, the reconstruction formula can be easily obtained by reviewing the argument between the equations \eqref{equ 4a1 in 1st example} and \eqref{2nd ordera1 integral id in Rn}. Formally we have 
	\[
	\widehat{\p_z^2 a}(\cdot,0)(-2 \xi)=-\int_{\p \Omega} \left.\frac{\p^2}{\p \eps_{1} \p \eps_{2}}\right|_{\eps_1=\eps_2=0}\Lambda_a\left(\eps_1 f_1 + \eps_2 f_2\right) \,dS,
	\]
which reconstructs the coefficient $\p_z^2a(x,0)$. Here $\widehat{\p_z^2 a}$ denotes the Fourier transformation of $a(x,z)$ in the $x$-variable, and $f_1$ and $f_2$ are the boundary values of the Calder\'on's exponential solutions~\eqref{calderon_exponential}. More precisely, we can take $f_1$ and $f_2$ to be the real or imaginary parts of the boundary values of the solutions~\eqref{calderon_exponential}, and we can then use a suitable combination as in Lemma \ref{lemma_complex_real_solutions} to recover $\widehat{\p_z^2 a}(\cdot,0)(-2 \xi)$. Moreover, by using \eqref{equ 4a1 in 1st example}, one can solve the boundary value problem \eqref{equ 4a1 in 1st example} uniquely to construct the function $w^{(kl)}$ given by \eqref{2nd ordera1 equal solutions}.

The proof for general $k$ is by recursion, but let us show separately how to reconstruct $\p_z^3a(x,z)$ corresponding to $k=3$.  To reconstruct $\p_z^3 a(x,0)$, we apply third order linearization for the equation 
\begin{align*}
\begin{cases}
\Delta u +a(x,u_j) =0 &\text{ in }\Omega,\\
u =\sum _{\ell=1}^{N+1} \eps_\ell f_\ell  &\text{ on }\p \Omega,
\end{cases}
\end{align*}
at $\eps=(\eps_1,\ldots,\eps_{N+1})=0$, where $\eps_\ell$ are small and $f_\ell\in C^s(\p \Omega)$. For $k=3$, we can take $N$ to be $2$. This shows that
	\begin{align}\label{equ 6a in 1st example}
	&\notag \Delta w^{(ik\ell)}
	+ \p ^2_{z}a(x,0) \left( w^{(ik)}v^{(\ell)} + w ^{(i \ell)}v^{(k)}  +w^{(k\ell )} v^{(i)} \right) \\
	&\quad + \p_z^3 a(x,0) \left( v^{(i)}v^{(k)}v^{(\ell)} \right) =0 & \text{ in }\Omega,
	\end{align}
	holds, where $w^{(ik\ell)}(x)=\frac{\p^3}{\p \eps_i\p \eps_k \p \eps_\ell}u(x;0)$. 
An integration by parts formula now yields that 
\begin{align*}
&\int_{\p \Omega} \left.\frac{\p^3}{\p \eps_i \p \eps_k \p \eps_\ell}\right|_{\eps_i=\eps_k=\eps_\ell =0}\Lambda_a \left(\eps_i f_i + \eps_k f_k + \eps_\ell  f_\ell \right) dS \\
 &\notag +\int_{\Omega }\p ^2_{z}a(x,0) \left( w^{(ik)}v^{(\ell)} + w ^{(i \ell)}v^{(k)}  +w^{(k\ell )} v^{(i)} \right)dx\ \\
= &\notag -\int_{\Omega} \p_z^3 a(x,0)v^{(i)}v^{(k)}v^{(\ell)}dx
\end{align*}
Let $v^{(i)}$ and $v^{(k)}$ be real or imaginary parts of Calder\'on's exponential solutions~\eqref{calderon_exponential} and let $v^{(\ell)}=1$. By using Lemma \ref{lemma_complex_real_solutions} and the fact that we have already reconstructed $\p_z^2 a(x,0)$ and $w^{(k\ell)}$, we see that we can reconstruct the Fourier transform of $\p_z^3 a(x,0)$. Consequently, we know the all the coefficients of the equation~\eqref{equ 6a in 1st example} for $w^{(ik\ell)}$, thus we may solve~\eqref{equ 6a in 1st example} to reconstruct also $w^{(ik\ell)}$. (The boundary value for $w^{(ik\ell)}$ is $0$.)

To reconstruct $\p_z^k a_j(x,0)$ for any $k\in \N$, one proceeds recursively. Let us assume that we have reconstructed $\p_z^ka(x,0)$ and $w^{(\ell_1\cdots \ell_{k})}$ for all $k\leq N$. The linearization of order $N+1$ then yields that (cf.~\eqref{equ 7a1 in 1st example})
%
\begin{align}\label{Nthreconstruction}
\begin{cases}
\Delta w^{(\ell_1\cdots \ell_{N+1})} +R_N(u,a,0)+ \p_z^{N+1} a(x,0)\left( \Pi _{k =1}^{N+1} v^{(\ell_k)} \right)=0 &\text{ in }\Omega,\\
w^{(\ell_1\cdots \ell_{N+1})}  =0 & \text{ on }\p \Omega,
\end{cases}
\end{align}
where $w^{(\ell_1\cdots \ell_{N+1})}(x)=\p^{N+1}_{\eps_{\ell_1} \cdots \eps_{\ell_{N+1}}}u(x,0)$ and where $R_N(u,a,0)$ is a polynomial of the functions $\p_z^{k}a(x,0)$ and $\p^{k}_{\eps_{\ell_1}\cdots \eps_{\ell_{k}}}u(x,0)$ for $k\leq N$. By the recursion assumption we have thus already recovered $R_N(u,a,0)$. Finally, integrating by parts shows that
\begin{align*}
	&\int_{\p \Omega} \frac{\p^{N+1}}{\p \eps_{\ell_1}\cdots \p \eps_{\ell_{N+1}}}\Bigg|_{\eps=0}\Lambda_a\left(\sum_{k=1}^{N+1} \eps_{\ell_k}f_{\ell_k} \right)dS + \int_\Omega R_N(u,a,0)\,dx \\
	= &-\int_\Omega  \p_z^L a(x,0)\left( \Pi _{\ell =1}^L v^{(i_\ell)} \right)dx.
\end{align*}
We choose $v^{(\ell_1)}$, $v^{(\ell_2)}$ to be real or imaginary parts of exponential solutions \eqref{calderon_exponential} and $v^{(\ell_3)}=\cdots=v^{(\ell_{N+1})}=1$ in $\Omega$. Using Lemma \ref{lemma_complex_real_solutions} again, this recovers $\p_z^{N+1} a(x,0)$. To end the reconstruction argument, we insert the now reconstructed $\p_z^{N+1} a(x,0)$ into~\eqref{Nthreconstruction} and solve the equation for $w^{(\ell_1\cdots \ell_{N+1})}$ with zero Dirichlet boundary value.
\end{proof}

\section{Simultaneous recovery of cavity and coefficients}\label{Section 3}
We prove Theorem \ref{thm: Nonlinear Schiffer's problem} by first recovering the cavity $D$ from the first linearization of the equation
\[
 \Delta u(x)+a(x,u)=0.
\]
After that the function $a=a(x,z)$ is recovered by higher order linearization.

\begin{proof}[Proof of Theorem \ref{thm: Nonlinear Schiffer's problem}]
	Let $f=\sum_{\ell=1}^{N+1}\eps_\ell f_\ell$, where $\eps_\ell$ are sufficiently small numbers and let $f_\ell \in C^s(\p \Omega)$ for all $\ell =1,2,\cdots, N +1$. We denote $\epsilon=(\eps_1 ,\ldots,\eps_{N+1})$ and let $u_j(x)=u_j(x;\epsilon)$ be the solution of 
	\begin{align}\label{Schiffer equation proof}
	\begin{cases}
		\Delta u_j +a_j (x,u_j)=0 & \text{ in }\Omega \setminus \overline{D_j}, \\
		u_j =0 & \text{ on }\p D_j,\\
		u_j =f & \text{ on }\p \Omega
	\end{cases}
	\end{align}
	with $f =\sum_{\ell=1}^{N+1}\eps_\ell f_\ell$, $j=1,2$.
	
	\vspace{10pt}
	
	{\it Step 1. Recovering the cavity.}\\
	
	\noindent  Let us differentiate \eqref{Schiffer equation proof} with respect to $\eps_{\ell}$, for $\ell =1,\cdots, N+1$. We obtain 
	\begin{align}\label{Schiffer 1st equ}
	\begin{cases}
	\Delta  \left( \frac{\p}{\p \eps_\ell}u_j \right)+\p_z a_j(x,u_j)\left( \frac{\p}{\p \eps_\ell}u_j \right)=0 & \text{ in }\Omega \setminus \overline{D_j},\\
    \frac{\p}{\p \eps_\ell}u_j =0 & \text{ on }\p D_j, \\
	 \frac{\p}{\p \eps_\ell}u_j =f_\ell & \text{ on }\p \Omega,
	\end{cases}
	\end{align}
	for all $\ell =1,2,\cdots, N+1$ and $j =1,2$. Note that by \eqref{condition of a}, the function $u_j(x;0)$ solves \eqref{Schiffer equation proof} with zero Dirichlet condition $\p \Omega$ and $\p D_j$. Thus we have $u_j(x;0)\equiv 0 $ in $\Omega \setminus \overline{D_j}$, for $j=1,2$. By letting $\epsilon=0$ and by denoting $v_j^{(\ell)}(x):= \frac{\p}{\p \eps_\ell}\big|_{\epsilon =0}u_j $, the equation \eqref{Schiffer 1st equ} becomes 
	\begin{align}\label{one-cavity problem}
	\begin{cases}
	\Delta  v_j^{(\ell)}=0 & \text{ in }\Omega \setminus \overline{D_j},\\
	v_j^{(\ell)} =0 & \text{ on }\p D_j, \\
	v_j^{(\ell)}=f_\ell & \text{ on }\p \Omega.
	\end{cases}
	\end{align}
	
	We show that $D_1=D_2$. This follows by a standard argument (see for instance \cite{beretta1998cavity, alessandrini2000optimal}), but we include a proof for completeness.
	Let $G$ be the connected component of $\Omega \setminus (\overline{D_1 \cup D_2})$ whose boundary contains $\p \Omega$ and let $\widetilde{v}^{(\ell)}:= v_1^{(\ell)}-v_2^{(\ell)}$. Then $\widetilde{v}^{(\ell)}$ solves 
	\begin{align*}
		\begin{cases}
		\Delta \widetilde{v}^{(\ell)} =0 & \text{ in } G, \\
		\widetilde{v}^{(\ell)}=\p_\nu \widetilde v^{(\ell)}=0 & \text{ on }\p \Omega
		\end{cases}
	\end{align*}
	since $\Lambda_{a_1}^{D_1}(f)= \Lambda_{a_2}^{D_2}(f) \text{ on }\p \Omega$ for small $f$. By the unique continuation principle for harmonic functions, one has that $\widetilde{v}^{(\ell)}=0$ in $G$. Thus 
 \begin{equation}\label{equalsontilde}
v_1^{(\ell)}=v_2^{(\ell)} \text{ in } G,
 \end{equation}
 for $\ell =1,\ldots, N+1$. In order to prove the uniqueness of the cavity, $D_1=D_2$, one needs only to consider the case $\ell=1$ of the problem \eqref{one-cavity problem}. However, we need to consider all $\ell =1,\cdots,N+1$ to recover the coefficient $a(x,z)$.
	
	We now argue by contradiction and assume that $D_1 \neq D_2$. The next step is to apply Lemma \ref{lemma_boundary_intersection} in the appendix with the choices $\Omega_1 = \Omega \setminus \overline{D_1}$, $\Omega_2 = \Omega \setminus \overline{D_2}$ and $\Gamma = \p \Omega$ (note that $\Omega_j$ and $\Gamma$ are connected by our assumptions). It follows, after interchanging $D_1$ and $D_2$ if necessary, that there exists a point $x_1$ such that 
	\[
	x_1\in \p G \cap (\Omega\setminus \overline{D_1}) \cap \p D_2.
	\]
	Since $x_1 \in \p D_2$, we have $v_2^{(\ell)}(x_1)=0$. By~\eqref{equalsontilde} and continuity, we also have that $v_1^{(\ell)}(x_1)=0$. The point $x_1$ is an interior point of the open set $\Omega\setminus \overline{D_1}$. Let us fix one of the boundary values $f_\ell$ to be non-negative and not identically $0$. 
	Now, since $v_1^{(\ell)}(x_1)=0$, the maximum principle implies that $v_1^{(\ell)}\equiv 0$ in the connected open set $\Omega \setminus \overline{D_1}$. This is in contradiction with the assumption that
	$v_1^{(\ell)}=f_\ell$ on $\p \Omega$  is non-vanishing (since $v_1^{(\ell)}$ is continuous up to boundary). This shows that $D:=D_1=D_2$. Moreover, we have by~\eqref{equalsontilde} that 
	\begin{equation}\label{cavity lins}
	v^{(\ell)}:=v_1^{(\ell)}=v_2^{(\ell)} \text{ in } \Omega \setminus \overline{D},
    \end{equation}
     for all $\ell =1,\cdots , N+1$ as desired. 
     
		\vspace{10pt}
	
	{\it Step 2. Recovering the coefficient.}\\
	
	\noindent  In order to prove the claim
	\begin{equation}\label{schiffer_claim}
	\p _z^k a_1(x,0) = \p _z^k a_2(x,0), \quad k \in \N
	 \end{equation}
	 of the theorem, we proceed by induction similar to the proof of Theorem \ref{thm uniqueness of a}. The equation~\eqref{schiffer_claim} is true for $k=1$ by assumption. Assume that~\eqref{schiffer_claim} holds for $k\leq N$, and assume also that
	 \begin{equation}\label{derivs up to N} 
	 \p^{k}_{\eps_{\ell_1}\cdots \eps_{\ell_{k}}}u_1(x,0)=\p^{k}_{\eps_{\ell_1}\cdots \eps_{\ell_{k}}}u_2(x,0)
	  \text{ for all } k \leq N.
	 \end{equation}
	 This equation holds for $k=1$ by~\eqref{cavity lins}.
	 
	By differentiating $N+1$ times the equation~\eqref{Schiffer equation proof} with respect to the parameters $\eps_{\ell_1},\ldots , \eps_{\ell_{N+1}}$ for $j=1,2$, and by subtracting the results from each other shows that in $\Omega\setminus \overline{D}$ one has 
	\begin{align}\label{equ 7 in 2nd example}
	&  \Delta \p^{N+1}_{\eps_{\ell_1}\cdots \eps_{\ell_{N+1}}}(u_1(x,0)-u_2(x,0))
	+ \p_z^{N+1} \left(a_1(x,0)-a_2(x,0)\right)\left( \Pi _{k =1}^{N+1} v^{(\ell_k)} \right)=0
	\end{align}
	 Here we used~\eqref{schiffer_claim} for $k \leq N$ and~\eqref{derivs up to N} to deduce that the terms with derivatives of order $\leq N$ vanish in the subtraction. We also have $\p^{N+1}_{\eps_{\ell_1}\cdots \eps_{\ell_{N+1}}}(u_1(x,0)-u_2(x,0))=0$ 
	 on $\p \Omega \cup \p D$.
	 
	 We know the DN map only on $\p \Omega$, but not on $\p D$. Therefore integrating~\eqref{equ 7 in 2nd example} and using integration by parts would produce an unknown integral over $\p D$. To compensate for the lack of knowledge on $\p D$, we proceed as follows. Let $v^{(0)}$ be the solution of 
	 \begin{align}\label{harmonic_function_partial_data}
	 	\begin{cases}
	 	\Delta v^{(0)}=0 & \text{ in }\Omega \setminus \overline{D}, \\
	 	v^{(0)}=0 & \text{ on }\p D, \\
	 	v^{(0)}=1 &\text{ on }\p \Omega. 
	 	\end{cases}
	 \end{align}
	By the maximum principle and by the fact that $\Omega \setminus \overline{D}$ is connected, we have that $v^{(0)}>0$ in $\Omega \setminus \overline{D}$.
	Multiplying the equation \eqref{equ 7 in 2nd example} by the function $v^{(0)}$, and then integrating the resulting equation yields  
	\begin{align}\label{integral id for generalaa}
	\begin{split}
	0= &  \int_{\p (\Omega \setminus \overline{D})}v^{(0)}\p _\nu \left(w_2^{(\ell_1\cdots\ell_{N+1})}-w_1^{(\ell_1\cdots\ell_{N+1})}\right)dS \\
	=& \int_{\Omega \setminus \overline{D}}v^{(0)}\Delta\left(w_2^{(\ell_1\cdots\ell_{N+1})}-w_1^{(\ell_1\cdots\ell_{N+1})}\right) dx \\
	& \quad+\int_{\Omega \setminus \overline{D}}\nabla v^{(0)}\cdot \nabla \left(w_2^{(\ell_1\cdots\ell_{N+1})}-w_1^{(\ell_1\cdots\ell_{N+1})}\right)dx\\
	= & \int_{\Omega\setminus\overline{D}}\p_z^{N+1} \left( a_1(x,0)-a_2(x,0) \right)\left( \Pi _{k =1}^{N+1} v^{(\ell_k)} \right)v^{(0)}\,dx \\
	& \quad+ \int_{\p (\Omega \setminus \overline{D})}\p _\nu v^{(0)}\left(w_2^{(\ell_1\cdots\ell_{N+1})}-w_1^{(\ell_1\cdots\ell_{N+1})}\right)dS \\
	= &\int_{\Omega\setminus\overline{D}}\p_z^{N+1} \left( a_1(x,0)-a_2(x,0) \right)\left( \Pi _{k =1}^{N+1} v^{(\ell_k)} \right)v^{(0)}\,dx,
	\end{split}
	\end{align}
	where we denoted $w_j^{(\ell_1\cdots\ell_{N+1})}(x)=\p^{N+1}_{\eps_{\ell_1}\cdots \eps_{\ell_{N+1}}}u_j(x,0)$ for $j=1,2$. In the first equality we used $v^{(0)}=0$ 
	on $\p D$ and the assumption that the DN maps agree on $\p \Omega$. In the second to last equality we used the fact that $v^{(0)}$ is harmonic. 
	
	Now, let us choose the boundary values as $f_3=f_4=\cdots=f_{N+1}=1$ on $\p \Omega$. With these boundary values the corresponding functions $v^{(\ell_k)}$ are harmonic functions in $\Omega \setminus \overline{D}$ with $v^{(\ell_k)}=1$ on $\p \Omega$ and $v^{(\ell_k)}=0$ on $\p D$. By the maximum principle, we have $0<v^{(\ell_k)}<1$ in $\Omega\setminus \overline{D}$ for $3\leq k \leq N+1$.
	By \cite[Theorem 1.1]{ferreira2009linearized} we can find special complex valued harmonic functions in $\Omega \setminus \overline{D}$ whose boundary values vanish on $\p D$ so that the products of pairs of these harmonic functions 
	form a complete subset in $L^1(\Omega)$. We use real and imaginary parts of these special harmonic functions as $v^{(\ell_1)}$ and $v^{(\ell_2)}$. From the integral identity \eqref{integral id for generalaa} and Lemma \ref{lemma_complex_real_solutions}, we conclude that $(q_1-q_2)v^{(0)}v^{(\ell_3)}v^{(\ell_4)}\cdots v^{(\ell_{N+1})}=0$ in $\Omega\setminus \overline{D}$. Since $v^{(\ell_k)}$ and $v^{(0)}$ are positive in $\Omega\setminus \overline{D}$ for $3\leq k \leq N+1$, this implies $q_1=q_2$ in $\Omega \setminus \overline{D}$ as desired.
\end{proof}

\section{Simultaneous recovery of boundary and coefficients}\label{Section 4}
We prove Theorem \ref{thm: partial data} by a similar method that we proved the Theorem~\ref{thm: Nonlinear Schiffer's problem}.

\begin{proof}[Proof of Theorem \ref{thm: partial data}]
%
We consider boundary data of the form $f=\sum_{\ell=1}^{N+1}\eps_\ell f_\ell$, where $\eps_\ell$ are small numbers and $f_\ell \in C^s_c(\Gamma)$ for all $\ell =1,\ldots, N+1$. Denote $\epsilon=(\eps_1 ,\cdots,\eps_{N+1})$. 
Let $u_j(x)=u_j(x;\epsilon)$, $j=1,2$, be the solution of
 \begin{align}\label{partial data equation}
	\begin{cases}
		\Delta u_j +a_j(x,u)=0 & \text{ in }\Omega_j, \\
		u_j =0 & \text{ on }\p \Omega_j\setminus \Gamma,\\
		u_j =f & \text{ on }\Gamma.
	\end{cases}
	\end{align}
	Note that by decreasing $\Gamma$ is necessary, we can assume that $\Gamma$ is connected.
	
	\vspace{10pt}
	
	{\it Step 1. Reconstruction of the boundary.}\\
	
	\noindent  
	By differentiating \eqref{partial data equation} with respect to $\eps_{\ell}$ for $\ell \in \N $, we obtain 
	\begin{align*}
		\begin{cases}
		\Delta  \left( \frac{\p}{\p \eps_\ell}u_j \right)+ \p_za(x,u_j) \left( \frac{\p}{\p \eps_\ell}u_j \right)=0 & \text{ in }\Omega_j,\\
		\frac{\p}{\p \eps_\ell}u_j =0 & \text{ on }\p \Omega_j\setminus \Gamma, \\
		\frac{\p}{\p \eps_\ell}u_j =f_\ell & \text{ on }\Gamma,
		\end{cases}
	\end{align*}
	for $j =1,2$. 
	By letting $\eps=0$ and using $u_j(x;0)=0$, we have that $v_j^{(\ell)}:= \left. \frac{\p}{\p \eps_\ell}u_j \right|_{\epsilon =0}$ solves: 
	\begin{align*}
		\begin{cases}
		\Delta  v_j^{(\ell)}=0 & \text{ in }\Omega_j,\\
		 v_j^{(\ell)} =0 & \text{ on }\p \Omega_j\setminus \Gamma, \\
		v_j^{(\ell)}=f_\ell & \text{ on }\Gamma,
		\end{cases}
	\end{align*}
	for $j =1,2$ and $ \ell =1,\cdots, N+1$.  
	Let $G$ be the connected component of $\Omega_1 \cap \Omega_2$ whose boundary contains the set $\Gamma$.
	Let $\widetilde{v}^{(\ell)}:=v_1^{(\ell)}-v_2^{(\ell)}$ in the domain  $G$. Then, using that $\Lambda_{a_1}^{\Omega_1,\Gamma}(f)= \Lambda_{a_2}^{\Omega_2, \Gamma}(f) \text{ on }\Gamma$ for small $f \in C^s_c(\Gamma)$, the function $\widetilde{v}^{(\ell)}$ solves  
	\begin{align*}
	\begin{cases}
	\Delta \widetilde{v}^{(\ell)} =0 & \text{ in } G,\\
	\widetilde{v}^{(\ell)} = \p _\nu \widetilde{v}^{(\ell)}=0 & \text{ on }\Gamma.
	\end{cases}
	\end{align*} 
	Then by the unique continuation principle for harmonic functions, we have that $\widetilde{v}^{(\ell)}=0$ in $G$. In other words, $v_1^{(\ell)}=v_2^{(\ell)}$ in $G$ for all $\ell =1,\cdots, N+1$. We remark that as in Section \ref{Section 3}, one only needs one harmonic function $v^{(1)}$ to recover the unknown boundary. For the coefficients, we still need many harmonic functions. Let us choose on the functions $f_\ell\in C_c^s(\Gamma)$ to be non-negative and not identically zero.
	
    If $\Omega_1\not =\Omega_2$, we can use Lemma \ref{lemma_boundary_intersection} in the appendix to conclude that (possibly after interchanging $\Omega_1$ and $\Omega_2$)
	there is a point $x_1$ with 
	\[
	x_1\in \p G \cap \Omega_1\cap (\p \Omega_2\setminus \Gamma).
	\]
	Since $x_1\in \p \Omega_2\setminus \Gamma$, it follows that
	$v_2^{(\ell)}(x_1)=0$. As $x_1$  is an interior point of the connected open set $\Omega_1$ and the boundary value of $v_2^{(\ell)}$  is non-negative,
	the maximum principle implies that $v_2^{(\ell)}\equiv0$ in $\Omega_1$. This is in contradiction with the assumption that
	$f_{\ell}$  is not identically zero. This shows that $\Omega_1=\Omega_2$. Furthermore, by denoting $\Omega := \Omega_1 = \Omega_2$, we have that $v^{(\ell)}:=v_1^{(\ell)}=v_2^{(\ell)}$ in $\Omega$ for $\ell =1,\cdots, N+1$.

	 \vspace{10pt}
	
	{\it Step 2. Reconstruction of  the coefficient.}\\
	
	\noindent The reconstruction of the Taylor series of $a(x,z)$ at $z=0$ is similar to Step $2$ in the proof of Theorem \ref{thm: Nonlinear Schiffer's problem}. First one shows by higher order linearization and by induction that the equation~\eqref{equ 7 in 2nd example} holds in $\Omega$. After that one constructs a harmonic function that vanishes on $\p \Omega \setminus \Gamma$ and which is positive on $\Gamma$. This is similar to the construction of $v^{(0)}$  in~\eqref{harmonic_function_partial_data}. The maximum principle shows that the constructed harmonic function is positive in $\Omega$. Integrating by parts as in~\eqref{integral id for generalaa} and using \cite[Theorem 1.1]{ferreira2009linearized} finishes the proof.
\end{proof}

\appendix

\section{} \label{sec_appendix}

Here we give a proof of a standard lemma (see e.g.\ \cite{beretta1998cavity}) that was used for recovering an unknown cavity or an unknown part of the boundary.

\begin{lem} \label{lemma_boundary_intersection}
Let $\Omega_1, \Omega_2 \subset \R^n$ be bounded connected open sets with $C^{\infty}$ boundaries, and assume that $\Gamma$ is a nonempty connected open subset of $\p \Omega_1 \cap \p \Omega_2$. Let $G$ be the connected component of $\Omega_1 \cap \Omega_2$ whose boundary contains $\Gamma$. If 
\[
\Omega_1 \neq \Omega_2,
\]
then, after interchanging $\Omega_1$ and $\Omega_2$ if necessary, one has 
\[
\p G \cap \Omega_1 \cap (\p \Omega_2 \setminus \Gamma) \neq \emptyset.
\]
\end{lem}
\begin{proof}
	
	Without loss of generality, we may assume that $\Omega_1\setminus {\Omega_2}\neq \emptyset$. We claim that we then have the inclusion relation 	\begin{align}\label{inclusion relation}
	\p (\Omega_1 \setminus G)\subset \left\{ \p G \cap  (\p \Omega_2\setminus \Gamma)  \right\}\cup (\p \Omega_1 \setminus \Gamma).
	\end{align}
	First, we prove \eqref{inclusion relation}. Using the fact that $\p E = \overline{E} \cap \left( \overline{\R^n \setminus E} \right)$ for any $E \subset \R^n$, and using that $A \setminus (B \setminus C) = (A \setminus B) \cup (A \cap C)$ and $\overline{A \cup B} = \overline{A} \cup \overline{B}$, one has 
	\begin{align*}
	\p (\Omega_1\setminus G)=& \left(\overline{\Omega_1 \setminus G} \right) \cap \left( \overline{\R^n \setminus (\Omega_1 \setminus G)} \right) \\
	 =& \left(\overline{\Omega_1 \setminus G} \right) \cap \left( \overline{(\R^n \setminus \Omega_1) \cup (\R^n \cap G)} \right) \\
	=& \left( \overline{\Omega_1 \setminus G} \cap \overline{\R^n \setminus \Omega_1} \right) \cup \left( \overline{\Omega_1 \setminus G}\cap \overline{G} \right) \\
	\subset & \left( \p \Omega_1 \setminus \Gamma  \right) \cup \left(   \p G \setminus \Gamma \right).
	\end{align*}
	Here we used that $(\overline{\Omega_1 \setminus G}) \cap \Gamma = \emptyset$. Next, one has $\ol{G} \cap (\Omega_1 \cap \Omega_2) \subset G$ (since any component of $\Omega_1 \cap \Omega_2$ that meets $\ol{G}$ must be equal to $G$), and thus we have 
	\[
	\p G = \ol{G} \setminus G \subset \ol{G} \setminus (\Omega_1 \cap \Omega_2) \subset \left( \ol{\Omega_1 \cap \Omega_2} \right) \setminus (\Omega_1 \cap \Omega_2) = \p (\Omega_1 \cap \Omega_2) \subset \p \Omega_1 \cup \p \Omega_2.
	\]
	It follows that $\p G\setminus \Gamma = \left\{ (\p \Omega_1 \cup \p \Omega_2)\cap \p G\right\}\setminus \Gamma$. Combining the above facts, we have proved \eqref{inclusion relation}.
	
	Next, by the above inclusion relation \eqref{inclusion relation}, it is easy to see that  
	\begin{align}\label{inclusion relation 2}
	\begin{split}
	\p (\Omega_1\setminus G)\cap \Omega_1 
	\subset &\left\{ \left(\p G \cap  (\p \Omega_2\setminus \Gamma)  \right) \cup (\p \Omega_1 \setminus \Gamma) \right\} \cap \Omega_1\\
	= & \left\{(\p G \cap  (\p \Omega_2\setminus \Gamma) \cap \Omega_1 \right\}\cup \left\{(\p \Omega _1\setminus \Gamma) \cap \Omega_1 \right\}\\
	=& \p G \cap \Omega_1 \cap (\p \Omega_2 \setminus \Gamma),
	\end{split}
	\end{align}
	where we have used that $\Omega_1$ is a bounded open set such that $(\p \Omega_1 \setminus \Gamma )\cap \Omega_1 = \emptyset$.
	
	We will now show that $\p G \cap \Omega_1 \cap (\p \Omega_2 \setminus \Gamma) \neq \emptyset$.
	Suppose that this is not true, i.e., $\p G \cap \Omega_1 \cap (\p \Omega_2 \setminus \Gamma) =  \emptyset$, then \eqref{inclusion relation 2} implies that 
	\begin{align}\label{emptyset}
		\p (\Omega_1\setminus G)\cap \Omega_1 =\emptyset.
	\end{align}
	Note that the following facts hold:
    \begin{gather}
    (\Omega_1\setminus G)\cap \Omega_1\neq \emptyset, \label{inclusion relation 3} \\
    \left\{ \R^n \setminus (\Omega_1 \setminus G)\right\}\cap \Omega_1\neq \emptyset. \label{inclusion relation 4}
    \end{gather}
	These facts are proved as follows.
    For \eqref{inclusion relation 3}, we have $(\Omega_1\setminus G)\cap \Omega_1 = \Omega_1 \setminus G$. If $\Omega_1 \setminus G=\emptyset$, we have $\Omega_1 \subset G$. However, by using the definition of $G$, we have that $G\subset \Omega_1 \cap \Omega_2\subset \Omega_1$, which implies that $\Omega_1=\Omega_1 \cap \Omega_2$.
    This violates our assumption that $\Omega_1\setminus \Omega_2 \neq \emptyset$. Thus we must have $\Omega _1 \setminus G \neq \emptyset$.
	Similarly, for \eqref{inclusion relation 4}, we can also obtain that 
	\begin{align*}
	\left\{\R^n \setminus (\Omega_1 \setminus G)\right\} \cap \Omega_1=\left\{ (\R^n\setminus \Omega_1)\cup G\right\} \cap \Omega_1=G \cap \Omega_1 \neq \emptyset.
	\end{align*}
	
	Finally, writing $V = \mathrm{int}(\Omega_1 \setminus G)$ and using \eqref{emptyset}--\eqref{inclusion relation 4}, we obtain that 
	\begin{gather*}
	V \cap \Omega_1 \neq \emptyset, \\
	(\R^n \setminus \ol{V}) \cap \Omega_1 \neq \emptyset.
	\end{gather*}
	Using \eqref{emptyset} again in the form $\p V \cap \Omega_1 = \emptyset$, we may decompose $\Omega_1$ as 
	\[
	\Omega_1= (V \cap \Omega_1) \cup \left\{ (\R^n \setminus \ol{V}) \cap \Omega_1  \right\}.
	\]
	Since $V$ is open, this implies that $\Omega_1$ can be written as the union of two nonempty disjoint open sets. This contradicts the assumption that $\Omega_1\subset \R^n$ is a connected set. Therefore, $\p G\cap \Omega_1 \cap (\p \Omega_2 \setminus \Gamma)$ must be a nonempty set, which completes the proof of Lemma \ref{lemma_boundary_intersection}.
	\end{proof}

\bibliographystyle{alpha}
\bibliography{ref}











\end{document}